\documentclass{amsart}

\input xy
\xyoption{all}

\usepackage{geometry}
\usepackage{amsmath}
\usepackage{amssymb}
\usepackage{cite}
\usepackage{amsthm}  
\usepackage{tikz}
\usetikzlibrary{er,positioning}

\newtheorem{same}{This should never appear}[section]
\newtheorem{defin}[same]{Definition}

\newtheorem{remark}[same]{Remark}
\newtheorem{theorem}[same]{Theorem}
\newtheorem{example}[same]{Example}
\newtheorem{lemma}[same]{Lemma}
\newtheorem{fact}[same]{Fact}
\newtheorem{question}[same]{Question}
\newtheorem{cor}[same]{Corollary}
\newtheorem{prop}[same]{Proposition}

\newtheorem{hypothesis}[same]{Hypothesis}

\newtheorem{defin*}{Definition}
\newtheorem*{theorem*}{Theorem}

\makeatletter
\newcommand{\skipitems}[1]{%
  \addtocounter{\@enumctr}{#1}%
}
\newcommand{\bb}{\mathbf{b}}
\newcommand{\rest}{\mathord{\upharpoonright}}
\newcommand{\id}{\textrm{id}}

\newcommand{\K}{\mathbf{K}}
\newcommand{\s}{\mathfrak{s}}

\newcommand{\Kp}{\K^{p\text{-grp}}}

\newcommand{\Kt}{\K^{\s\text{-Tor}}}
\newcommand{\Ks}{K^{\s\text{-Tor}}}
\newcommand{\Kto}{K^{R\text{-Tor}}}

\newcommand{\Kabt}{\K^{Tor}}

\newcommand{\LS}{\operatorname{LS}}

\newcommand{\leap}[1]{\le_{#1}}

\newcommand{\lea}{\leap{\K}}

\newcommand{\gtp}{\mathbf{gtp}}
\newcommand{\gS}{\mathbf{gS}}

\DeclareMathOperator{\cof}{cf}    

\title{A note on torsion modules with pure embeddings}

\author{Marcos Mazari-Armida}
\email{mmazaria@andrew.cmu.edu}
\urladdr{http://www.math.cmu.edu/~mmazaria/ }
\address{Department of Mathematical Sciences \\ Carnegie Mellon
University \\ Pittsburgh, Pennsylvania, USA}
\address{Current Address: Department of Mathematics \\ University of Colorado Boulder \\ Boulder, Colorado, USA}

\begin{document}

\begin{abstract}
We study Martsinkovsky-Russell torsion modules \cite{maru} with pure embeddings as an abstract elementary class. We give a model-theoretic characterization of the pure-injective and the $\Sigma$-pure-injective modules relative to the class of torsion modules assuming that the torsion submodule is a pure submodule.  Our characterization of relative $\Sigma$-pure-injective modules extends the classical charactetization of \cite{gj} and \cite[3.6]{zimm}.

We study the limit models of the class and determine when the class is superstable assuming that the torsion submodule is a pure submodule. As a corollary, we show that the  class of torsion abelian groups with pure embeddings is strictly stable, i.e.,  stable not superstable.

\end{abstract}


\maketitle

{\let\thefootnote\relax\footnote{{AMS 2020 Subject Classification:
Primary: 20K10, 03C48. Secondary: 03C45, 03C60, 13L05.
}}}


\section{Introduction}

Martsinkovsky-Russell torsion modules were introduced in \cite{maru} as a natural generalization of torsion modules to rings that are not necessarily commutative domains (Definition \ref{mdef}). We will denote them by \emph{$\s$-torsion modules} throughout this paper.  For a commutative domain, they are precisely the torsion modules, i.e., modules all of whose elements can be annihilated by a non-zero element of the ring.

For most rings the class of $\s$-torsion modules is not first-order axiomatizable in the language of modules. For example, it is folklore that the class of torsion abelian groups is not first-order axiomatizable. For this reason, we use non-elementary model-theoretic methods to analyse the class. More precisely, we will study the class of $\s$-torsion modules with pure embedding as an abstract elementary class (AEC for short, see Section 2.2 for more details). The class of $\s$-torsion modules with pure embedding is an abstract elementary class with amalgamation, joint embedding, and no maximal models. Moreover, it was shown in \cite[4.16]{maz2} that the class is stable. In this paper, assuming that the torsion submodule is a pure submodule, we study its class of limit models and use them to determine when the class is superstable. Recall that a \emph{limit model} is a universal model with some level of homogeneity (Definition \ref{10limit}) and an AEC is \emph{superstable} if there is a unique limit model up to isomorphims on a tail of cardinals (Definition \ref{defss}).\footnote{A detailed account of the development of the notion of superstability in AECs can be consulted in the introductions of \cite{grva} and \cite{m2}. }

A difficulty when trying to understand the class of $\s$-torsion modules is that the class might not be closed under pure-injective envelopes, see \cite[3.1]{m4} for the case of torsion abelian groups. Therefore, we begin by developing relative notions of pure-injectivity and $\Sigma$-pure-injectivity. The following result extends the  classical result of \cite{gj} and \cite[3.6]{zimm} where they characterize $\Sigma$-pure-injective modules (see Remark \ref{comp}).

\textbf{Lemma \ref{esigma}.}\textit{ Assume that $\s(N) \leq_p N$ for every module $N$ and $M$ is $\s$-torsion. 
$M$ is $\Sigma$-$\Ks$-pure-injective if and only if $M$ has the low-pp descending chain condition. }

The study of limit models for the class of $\s$-torsion modules and the characterization of superstability we obtain parallels that of previous results, \cite{m2}, \cite{m3} and \cite{maz2}, with the added difficulty that we have to deal with relative pure-injective modules instead of with pure-injective or cotorsion modules. More precisely, we obtain the following result.

\textbf{Theorem \ref{ss}.}\textit{ Assume that $\s(N) \leq_p N$ for every module $N$ and $R_R$ is not absolutely pure.  The following are equivalent.
\begin{enumerate}
\item The class of $\s$-torsion modules with pure embeddings is superstable.
\item There exists a $\lambda \geq (|R| + \aleph_0)^+$ such that the class of $\s$-torsion modules with pure embeddings has uniqueness of limit models of cardinality $\lambda$.
\item Every limit model in the class of $\s$-torsion modules with pure embeddings is $\Sigma$-$\Ks$-pure-injective.
\item Every $\s$-torsion module is $\Sigma$-$\Ks$-pure-injective.
\item Every $\s$-torsion module is $\Ks$-pure-injective.
\item For every $\lambda \geq |R| + \aleph_0$, the class of $\s$-torsion modules with pure embeddings has uniqueness of limit models of cardinality $\lambda$.
\item  For every $\lambda \geq |R| + \aleph_0$, the class of $\s$-torsion modules with pure embeddings is $\lambda$-stable.
\end{enumerate}}


The theorem allows us to show that certain classes are not superstable. In particular, we use our results to show that the class of torsion abelian groups with pure embeddings is strictly stable, i.e., stable not superstable. Determining if the class was superstable was the original objective of this paper.

\textbf{ Lemma \ref{sta}.} \textit{The class of torsion abelian groups with pure embeddings is $\lambda$-stable if and only if $\lambda^{\aleph_0}=\lambda$. Hence, it is strictly stable.}

This paper is part of a program to understand AECs of modules: \cite{maz20}, \cite{KuMaz}, \cite{m2}, \cite{m3}, \cite{m4}, \cite{maz2}. Other papers that have studied AECs of modules include:\cite{baldwine}, \cite{satr}, \cite{sh820}, \cite[\S 6]{bontor}\cite[\S 6]{lrv2}, \cite[\S 3]{lrvi}. 

The paper is divided into five sections. Section 2 has the preliminaries. Section 3 has new characterizations of relative pure-injective and $\Sigma$-pure-injective modules. Section 4 analyses the class of $\s$-torsion modules with pure embeddings as an abstract elementary class. Section 5 shows how to use the previous results to show that the class of torsion abelian groups with pure embeddings is strictly stable.

This paper was written while the author was working on a Ph.D. under the direction of Rami Grossberg at Carnegie Mellon University and I would like to thank Professor Grossberg for his guidance and assistance in my research in general and in this work in particular. I would also like to thank John T. Baldwin for asking me whether or not the class of torsion abelian groups with pure embeddings is superstable.  I thank Ivo Herzog for helpful conversations. I thank Hanif Cheung, Samson Leung, and Daniel Simson for many comments that helped improve the paper. I thank Mike Prest for letting me include Example \ref{E:1} in this paper. I am grateful to the referee for many comments that significantly improved the presentation of the paper.

\section{Preliminaries}

In this section we briefly present the basic notions of module theory and abstract elementary classes that we will use in this paper. The module theoretic preliminaries include the definition of the class of $\s$-torsion modules and assert some of its basic properties.

\subsection{Module theory} All rings considered in this paper are associative with unity. We write ${}_R M$ to specify that $M$ is a left $R$-module and $M_R$ to specify that $M$ is a right $R$-module.  When we simply write $M$, we assume that $M$ is a left $R$-module.

Given a ring $R$, $L_{R}= \{0, +,-\} \cup \{ r\cdot  : r \in R \}$ is the language of left $R$-modules where $r \cdot$ is interpreted as multiplication by $r$ for every $r \in R$. Recall that $\phi$ is a \emph{positive primitive formula}, $pp$-formula for short, if it is equivalent to an existentially quantified system of linear equations.  $M$ is a \emph{pure submodule} of $N$ if satisfaction of $pp$-formulas by elements from $M$ is the same whether considered in $M$ or in $N$ and we denote it by $M \leq_p N$.
The next family of $pp$-formulas was introduced in \cite{roth3}.

\begin{defin}\label{dlow}
A $pp$-formula $\psi(x)$ is \emph{low} if and only if $\psi[{}_R R]=0$. 
\end{defin}

\begin{remark}
It is easy to show that if $\psi_1(x), \psi_2(x)$ are low formulas and $r \in R$, then $\psi_1 + \psi_2(x):= \exists y \exists z ( \psi_1(y) \wedge \psi_2(z) \wedge x = y + z)$ and $r\psi_1(x):= \exists y( \psi(y)  \wedge x= ry)$ are low formulas.
\end{remark}

Given $\bb \in M^{<\omega}$ and $A \subseteq M$, the $pp$-type of $\bb$ over $A$ in $M$, denoted by $pp(\bb /A, M)$, is the set of all $pp$-formulas that hold for $\bb$ in $M$ with parameters in $A$.

As mentioned in the introduction, in this paper we will study the class of $\s$-torsion modules. These were introduced in \cite{maru} and studied from a model-theoretic perspective in \cite{maro} and \cite{roth2}. Below we present their model-theoretic definition.

\begin{defin}\label{mdef}
We say that $M$ is an \emph{$\s$-torsion module} if and only if for every $m \in M$ there is a low formula $\psi(x)$ such that $M \vDash \psi(m)$. We denote the class of $\s$-torsion modules by $\Ks$. 
\end{defin} 

\begin{remark}\label{R: tor} Let $R$ be a commutative domain. Recall that a module $M$ is an  $R$-torsion module if for every $m \in M$ there is an $r \neq 0 \in R$ such that $rm=0$. Denote the class of $R$-torsion modules by $\Kto$. It was shown in \cite[2.2]{maru} that $\Ks = \Kto$. In particular, the class of of $\s$-torsion abelian groups is precisely the class of torsion abelian groups. 
\end{remark} 

The following was introduced in \cite[2.1]{maru} and it also appears in a slightly different set up (under a different name) in \cite[5.3]{roth15}. The description we present will appear in the forthcoming paper \cite{maro}.

\begin{defin} For a left $R$-module $M$, let \[ \s(M) = \{ m \in M : M \vDash \psi(m) \text{ for some low formula } \psi \}. \]

\end{defin}

\begin{remark}\label{c-sum}\
\begin{itemize}
\item $M \in \Ks$ if and only if $\s(M)=M$.  
\item $\Ks$ is closed under pure submodules and direct sums.
\item (\cite[2.19]{maru}) $\s$ is a radical, i.e., for every $M, N$: $\s(M)$ is a submodule of $M$, if $f: M \to N$ then $f(\s(M)) \leq \s(N)$, and $\s(M/\s(M))=0$. 

\end{itemize}
\end{remark}
\begin{remark}
It is important to notice that in general $\s(\s(M))$ might be different from $\s(M)$, see \cite[p. 69]{maru}. For this reason, for arbitrary rings it might be the case that $\s(M)$ is \emph{not} an $\s$-torsion module.

\end{remark}

\subsection{Abstract elementary classes} We summarize the notions of abstract elementary classes that are used in this paper. A more detailed introduction to abstract elementary classes from an algebraic point of view is given in \cite[\S 2]{m4}. Given a model $M$, we write $|M|$ for its underlying set and $\| M \|$ for its cardinality. Abstract elementary classes were introduced by Shelah in \cite{sh88} to study those classes of structures that are axiomatizable in infinitary logics. An \emph{abstract elementary class} $\K$ is a pair $(K, \lea)$ where $K$ is a class of structures closed under isomorphisms and $\lea$ is a partial order on $K$. Additionally, the partial order on $K$ extends the substructure relation, $\K$ is closed under unions of chains, and there is a cardinal $\LS(\K)$ such that for any $M \in K$ and $A \subseteq |M|$, there is some $M_0 \lea M$ such that $A \subseteq |M_0|$ and $\|M_0\| \le |A| + \LS(\K)$. We consider only non-trivial AECs, i.e., AECs containing at least two non-isomorphic models. 

Throughout the paper, the term model refers to a structure in $K$.  For an infinite cardinal $\lambda$, we denote by $\K_\lambda$ the models in $\K$ of cardinality $\lambda$. If we write $f: M \to N$ for $M, N \in K$, we assume that $f$ is a $\K$-embedding unless specified otherwise. Recall that $f$ is a \emph{$\K$-embedding} if $f: M \cong f[M] \lea N$. Finally, for $M, N \in K$ and $A \subseteq M$, we write $f: M \xrightarrow[A]{} N$ if $f$ is a $\K$-embedding from $M$ to $N$ that fixes $A$ point-wise.

$\K$ has the \emph{amalgamation property} if for every $M, N_1, N_2$ in $K$ such that $M \lea N_1, N_2$ there are $N \in K$, $f: N_1 \to N$ and $g: N_2 \to N$ such that $f\rest_{M} = g\rest_{M}$. $\K$ has the \emph{joint embedding property} if every two models can be $\K$-embedded into a single model and $\K$ has \emph{no maximal models} if every model can be properly extended.

Shelah introduced a semantic notion of type in \cite{sh300}, we call them \emph{Galois-types} following \cite{grossberg2002}. Intuitively, a Galois-type over a model $M$ can be identified with an orbit of the group of automorphisms of the monster model which fixes $M$  point-wise. The full definition can be consulted in \cite[2.6]{m4}. We denote by $\gS(M)$ the set of all Galois-types over $M$ and we say that $\K$ is \emph{$(< \aleph_0)$-tame} if for every $M \in K$ and $p \neq q \in \gS(M)$, there is a finite subset $A$ of $M$ such that $p\rest_A \neq q\rest_A$.

\begin{defin}
 $\K$ is \emph{$\lambda$-stable} if $| \gS(M) | \leq \lambda$ for every $M \in \K_\lambda$. We say that $\K$ is \emph{stable} if $\K$ is $\lambda$-stable for some $\lambda \geq \LS(\K)$. 
\end{defin}

A model $M$ is \emph{universal over} $N$ if and only if $\| N\|= \| M\|=\lambda $ and for every $N^* \in \K_{\lambda}$ such that $N \lea N^*$, there is $f: N^* \xrightarrow[N]{} M$. 

\begin{defin}\label{10limit}
Let $\lambda$ be an infinite cardinal and $\alpha < \lambda^+$ be a limit ordinal.  $M$ is a \emph{$(\lambda,
\alpha)$-limit model over} $N$ if and only if there is $\{ M_i : i <
\alpha\}\subseteq \K_\lambda$ an increasing continuous chain such
that:
\begin{enumerate}
\item $M_0 =N$.
\item $M= \bigcup_{i < \alpha} M_i$.
\item $M_{i+1}$ is universal over $M_i$ for each $i <
\alpha$.
\end{enumerate}

$M$ is a 
\emph{$(\lambda, \alpha)$-limit model} if there is $N \in
\K_\lambda$ such that $M$ is a $(\lambda, \alpha)$-limit model over
$N$. $M$ is a \emph{$\lambda$-limit model} if there is a limit ordinal
$\alpha < \lambda^+$ such that $M$  is a $(\lambda,
\alpha)$-limit model. 
\end{defin}

\begin{fact}[{\cite[\S II]{shelahaecbook}, \cite[2.9]{tamenessone}}]\label{existence}
Let $\K$ be an AEC with joint embedding, amalgamation, and no maximal models. $\K$ is $\lambda$-stable if an only if $\K$ has a $\lambda$-limit model.
\end{fact}

A model is \emph{universal in $\K_\lambda$} if it has cardinality $\lambda$ and if every model in $\K$ of size $\lambda$ can be $\K$-embedded into it. It is known that every $\lambda$-limit model is universal in $\K_\lambda$ if $\K$ has the joint embedding property \cite[2.10]{maz20}.

We will also be interested in saturated models. Given $\lambda > \LS(\K)$ we say that $M$ is \emph{$\lambda$-saturated in $\K$} if every Galois-type over a $\K$-substructure of $M$ of size strictly less than $\lambda$ is realized in $M$. We have the following relation between saturated models and limit models.

\begin{fact}\label{sat} Let $\K$ be an AEC with joint embedding, amalgamation, and no maximal models.
If $M$ is a $(\lambda, \alpha)$-limit model and $\cof(\alpha) > \LS(\K)$, then $M$ is $\cof(\alpha)$-saturated in $\K$.
\end{fact}

We say that $\K$ has \emph{uniqueness of limit models of cardinality $\lambda$} if $\K$ has $\lambda$-limit models and if any two $\lambda$-limit models are isomorphic. 

\begin{defin}\label{defss}
$\K$ is \emph{superstable} if and only if $\K$ has uniqueness of limit models on a tail of cardinals. 
\end{defin}

Superstability was first introduced for AECs in \cite{sh394} and further developed in \cite{grva} and \cite{vaseyt}. There it is shown that for AECs that have amalgamation, joint embedding, no maximal models, and $\LS(\K)$-tameness, the definition above is equivalent to any other definition of superstability introduced for  AECs. In particular, for a complete first-order theory $T$, $(Mod(T), \preceq)$ is superstable if and only if $T$ is superstable as a first-order theory\footnote{$T$ is superstable if and only if $T$ is $\lambda$-stable for every $\lambda \geq 2^{|T|}$.}. 

Finally, we say that $\K$ is \emph{strictly stable} if $\K$ is stable and not superstable.


\section{Relative pure-injective and $\Sigma$-pure-injective modules}

In this section we extend classical results of pure-injective and $\Sigma$-pure-injective modules to the class of $\s$-torsion modules and show how these can be extended to other similar settings (see Remark \ref{comp}). The arguments are similar to the standard arguments, but we provide them to show that they come through in this non-first-order setting.

We will assume the following hypothesis for the rest of this section.

\begin{hypothesis}\label{10hyp2}
 $\s(N) \leq_p N$ for every left $R$-module $N$. 
\end{hypothesis}
\begin{remark}\label{R: imp}
In \cite[2.17]{maru} it was shown that if $R$ is right semihereditary\footnote{$R$ is right semihereditary if every finitely generated right submodule of a projective module is projective}, then Hypothesis \ref{10hyp2} holds.  Therefore, a more natural algebraic hypothesis would have been to assume that $R$ is right semihereditary. The reason we decided to assume Hypothesis \ref{10hyp2} instead is because it is potentially a weaker hypothesis and to emphasize how our results can be extended to similar settings (see Remark \ref{comp}).  An interesting question is to determine if both statements are equivalent. In the case of commutative domains,  a commutative domain is semihereditary if and only if it is a a Pr\"{u}fer domain. In that case, it is known that both assertions are equivalent \cite[Ex. 4.36]{lam07}.
\end{remark}


The next proposition follows directly from Hypothesis \ref{10hyp2}, but we record it due to its importance.

\begin{prop}\label{idem} \
\begin{enumerate}
\item $\s$ is idempotent, i.e., $\s(\s(M))=\s(M)$ for every left $R$-module $M$.
\item $\s(M) \in \Ks$ for every left $R$-module $M$. 
\item If $M \leq_p N$, then $\s(M) \leq_p \s(N)$. 
\end{enumerate}
\end{prop}

Recall that we are assuming that $\Ks$ is non-trivial. In \cite[2.32]{maru} it was shown, assuming Hypothesis \ref{10hyp2}, that $\Ks$ is non-trivial if and only if $R_R$ is not absolutely pure\footnote{$M_R$ is absolutely pure if for every $N_R$, if $M_R \subseteq_R N_R$ then $M_R \leq_p N_R$}.

A module $M$ is \emph{pure-injective} if for every $N_1, N_2$, if $N_1 \leq_p N_2$ and $f: N_1 \to M$  is a homomorphism then there is a homomorphism $g: N_2 \to M$ extending $f$.  Given a module $M$, \emph{the pure-injective envelope of $M$}, denoted by $PE(M)$, is a pure-injective module with $M \leq_p PE(M)$ and it is minimum with respect to these properties \cite[\S 3]{ziegler}, i.e., if $N$ is pure-injective and $M \leq_p N$ then there is a pure embedding $f: PE(M) \xrightarrow[M]{} N$. 

Let us recall the following notion and assertion.

\begin{defin}[{\cite[p. 145]{prest09}}]
Let $M \leq_p N$. $M$ is a \emph{pure-essential submodule} of $N$, denoted by $M \leq^e N$, if and only if for every homomorphism $f: N \to N'$, if $f \circ i$ is a pure embedding where $i: M \hookrightarrow N$ is the inclusion, then $f$ is a pure embedding.
\end{defin}

\begin{fact}[{\cite[4.3.15, 4.3.16]{prest09}}]\label{esse}\
\begin{enumerate}
\item If $M \leq_p N_1 \leq_p N_2$ and $M \leq^e N_2$, then $M\leq^e N_1$. 
\item  $M \leq^e PE(M)$.
\end{enumerate}
\end{fact}

We now introduce a relative notion of pure injectivity and saturation.

\begin{defin} Let $M$ be an $\s$-torsion module.
\begin{itemize}
\item $M$ is \emph{$\Ks$-pure-injective} if and only if for every $N_1, N_2 \in \Ks$, if $N_1 \leq_p N_2$ and $f: N_1 \to M$ is a homomorphism then there is a homomorphism $g: N_2 \to M$ extending $f$.
\item $M$ is \emph{$pp$-saturated in $N$} if and only if $M \leq_p N$ and if every $pp$-type over $M$ which is realized in $N$ is realized in $M$. 
\end{itemize}
\end{defin}

The following notion of partial homomorphism will also be useful and is due to \cite[\S 3]{ziegler}.

\begin{defin}
For two modules $M, N$, $A \subseteq |M|$ and $B \subseteq |N|$. A function $f: A \to B$ is a \emph{partial homomorphism from $M$ to $N$} if and only if for every $\bar{a} \in A$ and $\phi(\bar{x})$ $pp$-formula:

\[ M \vDash \phi(\bar{a}) \Rightarrow N \vDash \phi(f(\bar{a})).\]

\end{defin}

Observe that if $f: M \to N$ is a homomorphism then $f$ is a partial homomorphism from $M$ to $N$ as $pp$-formulas are preserved under homomorphism.


We now prove several equivalences of $\Ks$-pure-injectivity. These extend classical characterizations of pure-injectivity, see Remark \ref{comp} and the detailed history presented right before Theorem 2.8 of \cite{prest}.

\begin{theorem}\label{mod} Assume $M$ is an $\s$-torsion module.
The following are equivalent.
\begin{enumerate}
\item $M$ is $pp$-saturated in $N$ for every $N \in \Ks$.
\item $M$ is $\Ks$-pure-injective. 
\item $M=\s(PE(M))$
\item $M = \s(N)$ for some pure-injective module $N$. 
\item If $M \leq_ p M^*$ and $M^* \in \Ks$, then $M$ is a direct summand of $M^*$.
\end{enumerate}
\end{theorem}
\begin{proof}
$(1) \Rightarrow (2)$: Let $N_1 \leq_p N_2$ and $f: N_1 \to M$ be a homomorphism. Let \[\mathcal{P} = \{ g :  f\subseteq g \text{ and } g \text{ is a partial homomorphism from } N_2 \text{ to } M \}. \] It is clear that one can apply Zorn's lemma to $\mathcal{P}$, so let $g: A \to M$ be a maximal partial homomorphism from $N_2$ to $M$  extending $f$. We show that $A = N_2$. Let $b \in N_2$ and $p = pp(b/A, N_2)$. Consider $q(x)=\{ \phi(x, g(\bar{a})) : \phi(x, \bar{a}) \in p\}$. Clearly $q(x)$ is a $Th(M)$-type over $g[A]$ so there is $M^*$ an elementary extension of $M$ and $m^*\in M^*$ such that $q(x) \subseteq pp(m^*/M, M^*)$. 

Since $N_2 \in \Ks$, there is $\psi$ low such that $N_2 \vDash \psi(b)$. Hence $\psi \in q(x)$ and $m^* \in \s(M^*)$. Let $q'(x) = pp(m^*/M, \s(M^*))$. Then by (1), there is $m \in M$ realizing $q'(x)$ and it is clear that $g \cup \{ (b, m) \}$ is a partial homomorphism from $N_2$ to $M$  extending $f$. So by maximality of $g$ we have that  $b \in A$.

$(2) \Rightarrow (3)$: Let $N_1=M$, $N_2=\s(PE(M))$ and $f=\id_M$. Then by (2) there is a $g: \s(PE(M)) \to M$ extending $f$. Observe that by  Fact \ref{esse} $M \leq^e \s(PE(M))$ as $M \leq_p \s(PE(M)) \leq_p PE(M)$ and $M \leq^e PE(M)$. Then it follows that $g$ is a pure embedding.  We show that $\s(PE(M)) \subseteq M$. Let $a \in \s(PE(M))$ then $g(g(a))=g(a)$ as $g$ extends $f$. Since $g$ is injective $g(a)=a$, then $a = g(a) \in M$. Hence  $\s(PE(M)) = M$.

$(3) \Rightarrow (4)$: Clear.

$(4) \Rightarrow (5)$: Let $M \leq_ p M^*$ and $M^* \in \Ks$. Then by (4) we have that $M = \s(N) \leq_p N$ for $N$ a  pure-injective module. Since $N$ is pure-injective, there is a homomorphism $g: M^* \to N$ with $g\rest_M=\id_M$. One can check that $M^*= M \oplus ker(g)$ using that $g[M^*] \subseteq \s(N)=M$.

$(5) \Rightarrow (1)$: Let $M \leq_p N \in \Ks$ and $p=pp(a/M, N)$ for some $a \in N$. Then by (5) there is $L$ such that $N=M \oplus L$.  Let $\pi_1: N=M \oplus L \to M$ be the projection onto the first coordinate. Clearly $\pi_1 (a) \in M$ realizes $p$.
\end{proof}

A standard argument proves the following proposition.

\begin{prop}\label{10add} Assume $M$ and $N$ are $\s$-torsion modules. If $M$ and $N$ are $\Ks$-pure-injective, then $M \oplus N$ is $\Ks$-pure-injective
\end{prop}

We turn our attention to $\Sigma$-$\Ks$-pure-injective modules.

\begin{defin} Let $M$ be an $\s$-torsion module.
$M$ is \emph{$\Sigma$-$\Ks$-pure-injective} if and only if $M^{(I)}$ is $\Ks$-pure-injective for every index set $I$ where $M^{(I)}$ denotes the direct sum of $M$ indexed by $I$. 
\end{defin}

 \begin{defin}\label{d-chain} Let $M$ be an $\s$-torsion module. $M$ has \emph{the low-pp descending chain condition} if and only if  for every $\{ \phi_n(x) \}_{n \in \omega}$ such that $\phi_0(x)$ is low and $\phi_n(x)$ is a $pp$-formula for every $n \in \omega$, if $\{ \phi_n[M] \}_{n \in \omega}$  is a descending chain in $M$, then there exists $n_0 \in \omega$ such that $\phi_{n_0}[M]=\phi_k[M]$ for every $k \geq n_0$.

 \end{defin}
 
 We will soon see that the previous notion is actually equivalent to being $\Sigma$-$\Ks$-pure-injective

\begin{lemma}\label{pi-chain} Let $M$ be an $\s$-torsion module.
If $M$ has the low-pp descending chain condition, then $M$ is $\Ks$-pure-injective.   
\end{lemma}
\begin{proof}
Let $p=pp(b/M, N)$ for some $N \in \Ks$ with $b \in N$ and $M \leq_p N$. It is enough to show that there is a $\phi(x) \in p$ such that for every $\psi(x) \in p$ and $c \in M$, $M \vDash \phi(c) \to \psi(c)$. Such a $\phi$ exists by the hypothesis on $M$ and the fact that there is a low formula $\theta \in p$ as $N$ is an $\s$-torsion module. 
\end{proof}

The next result extends a classic characterization of $\Sigma$-pure-injectivity, see \cite[2.11]{prest} and Remark \ref{comp}.

\begin{lemma}\label{esigma}
Let $M$ be an $\s$-torsion module. The following are equivalent.
\begin{enumerate}
\item $M$ is $\Sigma$-$\Ks$-pure-injective.
\item $M^{(\aleph_0)}$ is $\Ks$-pure-injective.
\item $M$ has the low-pp descending chain condition.
\end{enumerate}
\end{lemma}
\begin{proof}
$(1) \Rightarrow (2)$: Clear.

$(2) \Rightarrow (3)$: Assume for the sake of contradiction that there is a family of $pp$-formulas $\{ \phi_n(x) \}_{n \in \omega}$ such that $\phi_0(x)$ is low and $\phi_n[M] \supset \phi_{n+1}[M]$ for every $n \in \omega$. For each $n \in \omega$ pick $a_n \in \phi_{n}[M] \backslash \phi_{n + 1}[M]$ and set $b_n=(a_0, a_1, \cdots , a_{n-1}, 0, \cdots ) \in M^{(\aleph_0)}$. 

Let $p(x)=\{ \phi_n(x - b_n) : n \geq 1 \} \cup \{ \phi_0(x) \}$. Realize that $p(x)$ is a $Th(M^{(\aleph_0)})$-type so there is $M^* \succeq M^{(\aleph_0)}$ and $c \in M^*$ realizing $p(x)$. Observe that $c \in \s(M^*)$ and that $M^{(\aleph_0)}$ is pure in $\s(M^*)$ because $\s(M^{(\aleph_0)})=M^{(\aleph_0)}$ and $\s(M^*)$ is pure in $M^*$. Then by hypothesis and Theorem \ref{mod}.(1) there is $d \in M^{(\aleph_0)}$  realizing $p(x)$. Then one can show that $M \vDash \phi_{m + 1}[a_m]$ for some $m \in \omega$, contradicting the choice of $a_m$.

$(3) \Rightarrow (1)$: As $pp$-formulas commute with direct sums, $M^{(I)}$ has the low-pp descending chain condition.  Since $M^{(I)} \in \Ks$, it follows that $M^{(I)}$ is $\Ks$-pure-injective by Lemma \ref{pi-chain}. \end{proof}

\begin{cor}\label{sigma}
Let $M$ and $N$ be $\s$-torsion modules .
\begin{itemize}
\item If $N$ is $\Sigma$-$\Ks$-pure-injective, then $N$ is $\Ks$-pure-injective.
\item If $M \leq_p N$ and $N$ is $\Sigma$-$\Ks$-pure-injective, then $M$ is $\Sigma$-$\Ks$-pure-injective.
\item If $M$ is elementarily equivalent to $N$ and $N$ is $\Sigma$-$\Ks$-pure-injective, then $M$ is $\Sigma$-$\Ks$-pure-injective.
\end{itemize}
\end{cor}

\begin{remark}\label{comp}
Let $\Phi$ be a collection of $pp$-formulas. We say that $M$ is an $\s_\Phi$-torsion module if and only if for every $m \in M$ there is a $\phi \in \Phi$ such that $M \vDash \phi(m)$. Given a module $M$, let  $\s_\Phi(M) = \{ m \in M : M \vDash \phi(m) \text{ for some } \phi \in \Phi \}$. This set up is similar to that of \cite[5.3]{roth15} and of the the forthcoming paper \cite{maro}.

If $\Phi$ is such that for every $N$ we have that $\s_\Phi(N) \leq_p N$ and the class of $\s_\Phi$-torsion modules is closed under finite direct sums, then $\s_\Phi(\text{-})$ is an idempotent radical and all the results we have proven so far hold for $\s_\Phi$-torsion modules. Where the low-pp descending change condition becomes the $\Phi$-descending chain condition.

Some examples of $\Phi$ include:
\begin{itemize}
\item $\Phi = \{ \phi : \phi \text{ is a low pp-formulas} \}$ under Hypothesis \ref{10hyp2}. Clearly the $\s_\Phi$-torsion modules are precisely the $\s$-torsion modules. 
\item  $\Phi = \{ x=x\}$. Clearly every module is an $\s_\Phi$-torsion module. Taking this $\Phi$, it follows that the results in this section extend the classical characterizations of pure-injective and $\Sigma$-pure-injective modules of  \cite{ste1}, \cite{kie}, \cite{war}, \cite{gj}, and \cite{zimm}. 
\item  $\Phi = \{ p^n x=0 : n < \omega \}$ for a fixed prime number $p$ in the ring of integers. Clearly the $\s_\Phi$-torsion modules are precisely the abelian $p$-groups. 
\item $\Phi = \{ \phi : \text{for every } M\in \mathcal{F} \text{, } M \vDash \forall x (\phi(x) \leftrightarrow x=0) \}$ where $\mathcal{F}$ is taken with respect to any torsion
theory $(\mathcal{T}, \mathcal{F})$ under the assumption that $\s_\Phi(M) \leq_p M$. The $\s_\Phi$-torsion modules are the elementary $\mathcal{T}$-torsion modules of \cite[5.6]{roth15}. 
\end{itemize}
\end{remark}

\section{$\s$-torsion modules as an AEC}

In this section we study the class of $\s$-torsion modules with pure embeddings as an abstract elementary class. There are three reasons why we decided to study $\s$-torsion modules with respect to pure embeddings rather than with respect to embeddings. Firstly, the class of $\s$-torsion modules is defined with respect to all low $pp$-formulas and not only those low quantifier-free formulas. Secondly, the class of $\s$-torsion modules is closed under pure submodules, but it is not necessarily closed under submodules. Finally, the original objective of this paper was to understand the class of torsion abelian groups with pure embeddings.

As in the previous section we are assuming that $\Ks$ is non-trivial.

\subsection{Basic properties} We begin by recalling some basic properties of the AEC of $\s$-torsion modules with pure embeddings. 

\begin{fact}
Let $R$ be a ring and $\Kt=(\Ks, \leq_p)$.
\begin{enumerate}
\item $\Kt$ is an AEC with $\LS(\Kt)=|R| + \aleph_0$.

\item $\Kt$ has amalgamation, joint embedding, and no maximal models.
\item If $\lambda^{|R| + \aleph_0} = \lambda$, then $\Kt$ is $\lambda$-stable.
\end{enumerate}
\end{fact}
\begin{proof}
(1) and (2) follow from \cite[4.2.(4), 4.8]{maz2} and (3) from \cite[4.16]{maz2}. 
\end{proof}

\begin{lemma}
$\Kt = (\Ks, \leq_p)$ is nicely generated in $(R\text{-Mod}, \leq_p)$ in the sense of \cite[4.1]{m4}, i.e., if $N_1, N_2 \in \Ks$ and $N_1, N_2\leq_p N$ for some module $N$, then there is $L \in \Ks$ such that $N_1, N_2 \leq_p L \subseteq N$.
\end{lemma}
\begin{proof}
 If $N_1, N_2 \in \Ks$ and $N_1, N_2\leq_p N$ for some module $N$, then $L = N_1 + N_2 \in \Ks$  and $N_1, N_2 \leq_p L \subseteq N$.
\end{proof}

The next result follows directly from the previous lemma, \cite[4.5]{m4}, and \cite[3.7]{KuMaz}.

\begin{cor}\label{gtp=pp} Let $N_1, N_2 \in \Ks$, $M \leq_p N_1, N_2$, $\bar{b}_1 \in N_1^{< \omega}$ and $ \bar{b}_2 \in N_2^{< \omega}$, then:
\[ \gtp_{\Ks}(\bar{b}_1/M; N_1) = \gtp_{\Ks}(\bar{b}_2/M; N_2) \text{ if and only if } pp(\bar{b}_1/M, N_1) = pp(\bar{b}_2/M, N_2).\]

In particular, $\Kt$ is $(<\aleph_0)$-tame.
\end{cor}

The next result follows from the previous lemma and \cite[4.6]{m4}.

\begin{cor} Let $\lambda \geq |R| + \aleph_0$.
If $(R\text{-Mod}, \leq_p)$ is $\lambda$-stable, then $(\Ks, \leq_p)$ is $\lambda$-stable.
\end{cor}

For an arbitrary ring $R$,  $(R\text{-Mod}, \leq_p)$ is $\lambda$-stable if $\lambda^{|R| + \aleph_0}=\lambda$ by \cite[3.16]{KuMaz}. Moreover, if $R$ is left pure-semisimple, then  $(R\text{-Mod}, \leq_p)$ is $\lambda$-stable for every $\lambda \geq |R| + \aleph_0$ by \cite[4.28]{m2}.

\subsection{Limit models and superstability} We characterize limit models algebraically and use them to characterize superstability.

We assume \emph{Hypothesis \ref{10hyp2}} for the rest of this section, i.e., we assume that $\s(N) \leq_p N$ for every left $R$-module $N$. 

\begin{lemma}
If $M$ is $(|R| + \aleph_0)^+$-saturated in $\Kt$, then $M$ is $\Ks$-pure-injective.
\end{lemma}
\begin{proof}
We use Theorem \ref{mod}. Let $M \leq_p N \in \Ks$ and $p=pp(b/M, N)$ for some $b \in N$. Given $\phi(x, \bar{y})$, a $pp$-formula, let $A_\phi =\{ \bar{m} \in M : \phi(x,\bar{m}) \in p \}$ and let $\bar{m}_\phi$ be an element of $A_\phi$ if $A_\phi \neq \emptyset$ and $\bar{m}_\phi = \bar{0}$ otherwise. Let $B = \bigcup_{\phi(x, \bar{y}) \in pp\text{-formula}} \bar{m}_\phi$ and $M^*$ be a structure obtained by applying the downward L\"{o}wenheim-Skolem axiom to $B$ in $M$. Observe that $\| M^* \| = |R| + \aleph_0$.

Let $q=\gtp(b/M^*; N)$. Since $M$ is $(|R| + \aleph_0)^+$-saturated in $\Kt$ there is $c \in M$ such that $q=\gtp(c/M^*; M)$.  Then $pp(c/M^*, M)=pp(b/M^*, N)$ by Lemma \ref{gtp=pp}. Using that $pp$-formulas define cosets \cite[2.2]{prest} and the choices of the $\bar{m}_\phi$'s, it follows that $c$ realizes $p$. 
\end{proof}

Using Fact \ref{sat}, we derive:

\begin{cor}\label{blim}
If $M$ is a $(\lambda, \alpha)$-limit model in $\Kt$ and $\cof(\alpha) \geq (|R| + \aleph_0)^+$, then $M$ is $\Ks$-pure-injective.
\end{cor}

We are going to show that limit models with \emph{long chains} are isomorphic. In order to do that, we obtain a couple of algebraic results regarding pure-injective modules. We begin by generalizing the following fact.

\begin{fact}[{\cite[2.5]{gks}}]\label{10ipi}
 Let $M, N$ be pure-injective modules. If there are pure embeddings $f: M \to N$ and $g: N \to M$, then $M$ and $N$ are isomorphic. 
\end{fact}

\begin{lemma}
Let $M, N$ be any two modules. If there are pure embeddings $f: M \to N$ and $g: N \to M$, then $PE(M)$ and $PE(N)$ are isomorphic. 
\end{lemma}
\begin{proof}
It is enough to show that there are $f': PE(M) \to PE(N)$ and $g': PE(N) \to PE(M)$ pure embeddings, as then the result follows directly from Fact \ref{10ipi}. The existence of $f'$ and $g'$ follow from the minimality of $PE(M)$ and $PE(N)$ respectively. 
\end{proof}

\begin{cor}
Let $M, N$ be $\s$-torsion and $\Ks$-pure-injective modules. If there are pure embeddings $f: M \to N$ and $g: N \to M$, then $M$ and $N$ are isomorphic.
\end{cor}
\begin{proof}
Follows directly from the previous result and Theorem \ref{mod}.(3).
\end{proof}

Since $\lambda$-limit models are universal in $(\Kt)_\lambda$, we obtain the following result.

\begin{cor}\label{uni} Assume $\lambda\geq |R| + \aleph_0$.
If $M, N$ are $\lambda$-limit models in $\Kt$ and $\Ks$-pure-injective, then $M$ and  $N$ are isomorphic. 
\end{cor}

Putting together the above assertion with Lemma \ref{blim}, we obtain:

\begin{cor}\label{longuni} Assume $\lambda\geq (|R| + \aleph_0)^+$.
 If $M$ is a $(\lambda, \alpha)$-limit model in $\Kt$ and $N$ is a $(\lambda,
\beta)$-limit model in $\Kt$ such that $\cof(\alpha), \cof(\beta) \geq (|R| + \aleph_0)^+$,
then $M$ and  $N$ are isomorphic. 
\end{cor}

A standard argument can be used to obtain the following assertion from Theorem \ref{mod}.(5) and Proposition \ref{10add}. See for example \cite[4.5, 4.6]{KuMaz}.

\begin{lemma}\label{ccountablelim}  Assume $\lambda\geq
(|R| + \aleph_0)^+$. If $M $ is
a $(\lambda, \omega)$-limit model in $\Kt$ and $N $ is a
$(\lambda, (|R| + \aleph_0)^+)$-limit model in $\Kt$, then $M$ and  $N^{(\aleph_0)}$ are isomorphic.

\end{lemma} 

We also have that limit models are elementarily equivalent. The argument of \cite[4.2]{KuMaz} can be used in this setting as the class has the joint embedding property.

\begin{lemma}\label{10elem}
If $M$ and $N$ are limit models in $\Kt$, then $M$ and $N$ are elementarily equivalent. 
\end{lemma}

\begin{theorem}\label{ss}
Assume Hypothesis \ref{10hyp2} and that $R_R$ is not absolutely pure.  The following are equivalent.
\begin{enumerate}
\item $\Kt$ is superstable.
\item There exists a $\lambda \geq (|R| + \aleph_0)^+$ such that $\Kt$ has uniqueness of limit models of cardinality $\lambda$.
\item Every limit model in $\Kt$ is $\Sigma$-$\Ks$-pure-injective.
\item Every $M \in \Ks$ is $\Sigma$-$\Ks$-pure-injective.
\item Every $M \in \Ks$ is $\Ks$-pure-injective.
\item For every $\lambda \geq |R| + \aleph_0$, $\Kt$ has uniqueness of limit models of cardinality $\lambda$.
\item  For every $\lambda \geq |R| + \aleph_0$, $\Kt$ is $\lambda$-stable.
\end{enumerate}
\end{theorem}
\begin{proof}
$(1) \Rightarrow (2)$: Clear.

$(2) \Rightarrow (3)$: The proof is similar to that of (2) to (3) of \cite[3.15]{maz2}. The reason that argument goes through in $\Kt$ is because of Lemma \ref{ccountablelim}, Lemma \ref{blim}, Lemma \ref{esigma}, Lemma \ref{10elem}, and Corollary \ref{sigma}.

$(3) \Rightarrow (4)$: Follows from Corollary \ref{sigma} and the fact that limit models are universal.

$(4) \Rightarrow (5)$: Follows from Corollary \ref{sigma}.

$(5) \Rightarrow (6)$: By Corollary \ref{uni} for every cardinal  $\lambda \geq |R| + \aleph_0$ there is at most one $\lambda$-limit model up to isomorphisms, so we only need to show existence. We show that $\Kt$ is $\lambda$-stable for every $\lambda \geq |R| + \aleph_0$, this is enough by Fact \ref{existence}. Let $\lambda \geq |R| + \aleph_0$ and $M \in \Kt_\lambda$.

Let $N \in \Kt$ and $\{ a_i : i < \kappa \} \subseteq N$ such that $\{ \gtp(a_i/ M; N) : i < \kappa \}$ is an enumeration without repetitions of $\gS(M)$. Let $\Delta :=\{ pp(a_i/M, N) : i < \kappa \}$ and observe that $| \gS(M) | \leq | \Delta |$ since $\Phi: \gS(M) \to \Delta$ given by $\Phi(\gtp(a_i/ M; N))=pp(a_i/M, N)$ is injective by Lemma \ref{gtp=pp}.

Since $N$ is $\Sigma$-$\Ks$-pure-injective by $(5)$, it follows from Lemma \ref{esigma} that $N$ has the low-pp descending chain condition. Then it follows, as in Lemma \ref{pi-chain}, that for every $p \in \Delta$ there is $\psi_{p} \in p$ such that for every $\theta \in p$ and $c \in N$, $N \vDash \psi_p(c) \to \theta(c)$. Let $\Psi: \Delta \to \{ \phi(x, \bar{m}) : \phi(x, \bar{y}) \text{ is a }pp\text{-formula and } \bar{m} \in M \}$ be given by $\Psi(p)=\psi_p$. It is easy to show that $\Psi$ is injective and as $|\{ \phi(x, \bar{m}) : \phi(x, \bar{y}) \text{ is a }pp\text{-formula and } \bar{m} \in M \}|=(|R| + \aleph_0)\lambda=\lambda$, we can conclude that $|\Delta | \leq \lambda$. Therefore,  $| \gS(M) | \leq \lambda$.

$(6) \Rightarrow (1)$: Clear. 

$(6) \Rightarrow (7)$: Clear.

$(7) \Rightarrow (4)$: Assume for the sake of contradiction that there is $M \in \Ks$ which is not $\Sigma$-$\Ks$-pure-injective. It follows from Lemma \ref{esigma} that there is a set of formulas $\{ \phi_n(x) \}_{n \in \omega}$ such that $\phi_0(x)$ is low and $\phi_n(x)$ is a $pp$-formula for every $n \in \omega$ such that $\phi_n[M] \supset \phi_{n+1}[M]$ for every $n \in \omega$.

Let $\lambda = \beth_\omega( |R| + \aleph_0)$. Observe that since $[ \phi_n[M] : \phi_{n+1}[M] ] \geq 2$, it follows that $[ \phi_n[M^{(\lambda)}] : \phi_{n+1}[M^{(\lambda)}] ] = \lambda$ for each $n \in \omega$ and $M^{(\lambda)} \in \Ks$ by Remark \ref{c-sum}. For every $n\in \omega$ pick $\{ a_{n, \alpha} :  \alpha < \lambda \} \subseteq  M^{(\lambda)}$ a complete set of representatives of $\phi_n[M^{(\lambda)}]/ \phi_{n+1}[M^{(\lambda)}]$.

 Let $A = \bigcup_{n < \omega} \{ a_{n, \alpha} :  \alpha < \lambda \}$ and $N$ be a structure obtained by applying the downward L\"{o}wenheim-Skolem Theorem to $A$ in $M^{(\lambda)}$. It is clear that $N \in \Kt_\lambda$. For every $\eta \in \lambda^\omega$, let $\Phi_\eta=\{ \phi_{n+1}(x - \Sigma_{i=0}^n a_{i, \eta(i)}) : n < \omega \} \cup \{ \phi_0(x)\}$. $\Phi_\eta$ is a $Th(M^{(\lambda)})$-type over $N$, so pick $M_\eta \succeq M^{(\lambda)}$ and $c_{\eta} \in M_\eta$ realizing $\Phi_\eta$. It is clear that $c_\eta \in \s(M_\eta)$ so consider $q_\eta= \gtp(c_\eta/N; \s(M_\eta))$. 
 
 Using Lemma \ref{gtp=pp} and that $\s(M_\eta)\leq_p M_\eta$ for every $\eta \in \lambda^\omega$, it can be shown that if $\eta_1 \neq \eta_2 \in \lambda^\omega$ then $q_{\eta_1} \neq q_{\eta_2}$. Hence $| \gS(N) | \geq \lambda^{\aleph_0} > \lambda$ by the choice of $\lambda$ and K\"{o}nig's lemma. This contradicts our assumption that $\Kt$ was $\lambda$-stable.  \end{proof}
 
 \begin{remark}
 The equivalence between $(4)$ and $(7)$ of the above theorem is a natural extension of a result of Garavaglia and Macintyre \cite[Theo 1]{gar}, \cite[3.11]{prest}.
 \end{remark}
 
 We give an example of a ring satisfying the conditions of Theorem \ref{ss}. This example is due to M. Prest. 
 
 \begin{example}\label{E:1}
 Let $R$ be the ring of $n \times n$ upper triangular matrices over a countable field $k$ with $n \geq 2$. $R$ is left and right hereditary \cite[2.36]{lam2}, thus Hypothesis \ref{10hyp2} holds by Remark \ref{R: imp}. $R_R$ is not absolutely pure as that would imply that $R$ is right self-injective because $R$ is right pure-semisimple \cite[4.5.18]{prest09}, but $R$ is not right self-injective \cite[3.10B]{lam2}. $R$ is left pure-semisimple \cite[4.5.18]{prest09}, therefore every left module is pure-injective and condition (5) of Theorem \ref{ss} holds. 
 \end{example}
Previous results that characterised superstability in classes of modules always corresponded to classical rings \cite{m2}, \cite{m3}, \cite{maz2}. In this case we do not know if that is the case. If $R$ is a left pure-semsimple ring such that $R$ is right semihereditary and $R_R$ is not absolutely pure, then condition (5) of Theorem \ref{ss} holds. For this reason, we think of a ring satisfying any of the conditions given in Theorem \ref{ss} as a \emph{weak} pure-semisimple ring. A natural question to ask is the following:

\begin{question}\label{q1}
Assume $R$ is right semihereditary, $R_R$ is not absolutely pure and any of the equivalent conditions of Theorem \ref{ss} hold. Is $R$ left pure-semisimple?
\end{question}
 
Even if superstability for $\s$-torsion modules with pure embeddings does not correspond to a classical ring, Theorem \ref{ss} is still interesting as it can be used to show that certain classes are not superstable. An example of this is given in the next section. 

\begin{remark}
Assume $R$ is a Pr\"{u}fer domain.\footnote{Recall that the class of $\s$-torsion modules coincides with the class of $R$-torsion modules by \cite[2.2]{maru}.} Recall the following notion introduced in \cite[XIII.\S 6]{fl}. $M$ is \emph{torsion-ultracomplete} if for every module $N$, if $M\leq_p N$ and $N/M \in \Ks$, then $M$ is a direct summand of $N$. It follows from Theorem \ref{mod}, that an $\s$-torsion module is  torsion-ultracomplete if and only if it is $\Ks$-pure-injective. A finer question than Question \ref{q1} would be to determine if there is a Pr\"{u}fer domain which is not a field such that every torsion module is torsion-ultracomplete.
\end{remark}




Finally, a natural question is if any of the results presented in this section can be extended to rings where Hypothesis \ref{10hyp2} fails. We think it is unlikely. However, we think that if one studies the class of $\s$-torsion modules with respect to other embeddings it is possible to obtain  analogous results to the ones presented here for rings that not necessarily satisfy Hypothesis \ref{10hyp2}.

\section{Torsion abelian groups}

In this section we apply our general results to the class of torsion abelian groups with pure embeddings. 
We show it is strictly stable, characterize its stability cardinals, and describe its limit models. We will denote the class of torsion abelian groups with pure embeddings by $\Kabt$.

By Remark \ref{R: tor} and the fact that the torsion subgroup is a pure subgroup, the previous results apply to $\Kabt$ where we write $t(G)$ for $\s(G)$. 

The following fact collects what is known of the class of torsion abelian groups with pure embeddings. They were first obtained in \cite[\S 4]{m4}, but they also follow from the results of the previous section.

\begin{fact}\label{basic}  Let $\Kabt = (K^{\text{Tor}}, \leq_p)$.
\begin{itemize}
\item $\Kabt$ is an AEC with $\LS(\Kabt)=\aleph_0$ that has amalgamation, joint embedding, and no maximal models.
\item If $\lambda^{\aleph_0}=\lambda$, then $\Kabt$ is $\lambda$-stable.
\item $\Kabt$ is $(<\aleph_0)$-tame.
\end{itemize}
\end{fact}

We will use the following algebraic result to show that the class is not superstable.

\begin{remark}[{\cite[\S 10.3]{fuchs}}]\label{rem-ab}
Let $B_n= \mathbb{Z}(p^n)^{(\lambda)}$ and $B=\bigoplus_n B_n$. 
\begin{equation}
g=(b_n)_{n \in \omega} \in t(PE(B))\leq \Pi_n B_n \text{ if and only if the orders of } \{b_n\}_{n \in \omega} \text{ are bounded.}
\end{equation}

Thus $\| t(PE(B)) \| = \lambda^{\aleph_0}$ as $|B_n[p]|= |\{ b \in B_n : pb =0 \} | = \lambda$ for every $n \in \omega$.

\end{remark}

\begin{lemma}\label{nss}
$\Kabt$ is not superstable. Hence, $\Kabt$ is strictly stable.
\end{lemma}
\begin{proof}
Assume for the sake of contradiction that $\Kabt$ is superstable.  Let $\lambda= \beth_{\omega}$ and $B=\bigoplus_n B_n$ where $B_n= \mathbb{Z}(p^n)^{(\lambda)}$  for every $n < \omega$ as in Remark \ref{rem-ab}. Then by Theorem \ref{ss}.(5) and Theorem \ref{mod}.(3), it follows that $B = t(PE(B))$. This is a contradiction as $\| t(PE(B)) \| = \lambda^{\aleph_0} > \lambda$ by K\"{o}nig's lemma.
\end{proof}

\begin{remark}
The previous result contrasts with the fact that the class of torsion abelian groups with embedding  is superstable \cite[4.8]{m4}.
\end{remark}

We are actually able to obtain a complete characterization of the stability cardinals.

\begin{lemma}\label{sta}
$\Kabt$ is $\lambda$-stable if and only if $\lambda^{\aleph_0}=\lambda$.
\end{lemma}
\begin{proof}
The backward direction follows from Fact \ref{basic} so we show the forward direction. We divide the proof into two cases:

\underline{Case 1:} $\lambda > \aleph_0$. Assume that $\Kabt$ is $\lambda$-stable. Let $M$ be a $(\lambda, \omega_1)$-limit model and $B=\bigoplus_n B_n$ where $B_n= \mathbb{Z}(p^n)^{(\lambda)}$  for every $n < \omega$ as in Remark \ref{rem-ab}.  Since $M$ is a $\lambda$-limit model and $B$ has size $\lambda$ there is a pure embedding $f: B \to M$.  Then there is $g: PE(B) \to PE(M)$ pure embedding extending $f$ by the minimality of $PE(B)$.

In particular, $g\rest_ {t(PE(B))}:  t(PE(B)) \to t(PE(M))$ is injective. So $\| t(PE(B)) \| \leq \| t(PE(M)) \|$. Since $t(PE(M))=M$ by Corollary \ref{blim} and Theorem \ref{mod} and $\| t(PE(B)) \| = \lambda^{\aleph_0}$ by Remark \ref{rem-ab}, it follows that $\lambda = \lambda^{\aleph_0}$

\underline{Case 2:} $\lambda = \aleph_0$. Assume for the sake of contradiction that $\Kabt$ is $\omega$-stable. Since $\Kabt$ is $(<\aleph_0)$-tame by Fact \ref{basic}, it follows from \cite[3.6]{bkv} that $\Kabt$ is  $\beth_\omega$-stable. This contradicts the previous case as $\beth_\omega^{\aleph_0} > \beth_\omega$ by K\"{o}nig's lemma.
\end{proof}

From the above results we can precisely describe the spectrum function for limit models.
\begin{cor}
 If $\lambda^{\aleph_0}= \lambda$, then $\Kabt$ has two non-isomorphic $\lambda$-limit models. For every other cardinal $\lambda$, $\Kabt$ has no $\lambda$-limit models.
\end{cor}
\begin{proof}
Assume $\lambda^{\aleph_0}= \lambda$. Then the $(\lambda, \omega)$-limit model and the $(\lambda, \omega_1)$-limit model are not isomorphic by Corollary 4.11 and Theorem \ref{ss}.(2). Moreover, given $M$ a $(\lambda, \alpha)$-limit model, $M$ is isomorphic to the $(\lambda, \omega)$-limit model if $\cof(\alpha) =\omega$ by a back-and-forth argument and to the $(\lambda, \omega_1)$-limit model if $\cof(\alpha) > \omega$ by Corollary 4.11. The second part follows from Lemma \ref{sta} and Fact \ref{existence}.
\end{proof}

We go one step further and give an algebraic description of the limit models.  Recall that given $n \in \mathbb{N}$ and $G$ an abelian group, $G[n]$ denotes the elements annihilated by $n$ in $G$ and $nG$ denotes the elements of the form $ng$ for some $g$ in $G$.

\begin{lemma}
Let $\lambda$ be an infinite cardinal such that $\lambda^{\aleph_0}=\lambda$ and $\alpha < \lambda^+$ be a limit ordinal. If $M$ is a $(\lambda, \alpha)$-limit model in $\Kabt$, then:
\begin{enumerate}
\item If $\cof(\alpha) >  \omega$, then $M \cong t(\Pi_p PE(\bigoplus_n \mathbb{Z}(p^n)^{(\lambda)})) \oplus \bigoplus_p \mathbb{Z}(p^\infty)^{(\lambda)}$.

\item If $\cof(\alpha)=\omega$, then $M \cong t(\Pi_p PE(\bigoplus_n \mathbb{Z}(p^n)^{(\lambda)}))^{(\aleph_0)} \oplus \bigoplus_p \mathbb{Z}(p^\infty)^{(\lambda)}$.
\end{enumerate}

\end{lemma}
\begin{proof}
We begin with $(1)$. By Corollary \ref{blim} and Theorem \ref{mod} we have that $M = t(G)$ for some pure-injective group $G$.  By \cite[\S 1]{ef}:

\[ G=\Pi_p PE(\bigoplus_n \mathbb{Z}(p^n)^{(\alpha_{p, n})} \oplus \mathbb{Z}_p^{(\beta_p)} ) \oplus \mathbb{Q}^{(\delta)} \oplus (\bigoplus_p \mathbb{Z}(p^\infty)^{(\gamma_p)}), \]

for some specific $\alpha_{p, n}$, $\beta_p$, $\delta$, $\gamma_p$ described in \cite[\S 1]{ef} for every $p$ a prime number and $n < \omega$. As $M$ is a torsion group, we may assume that $\beta_p = \delta =0$ for every $p$. Hence we only need to determine $\alpha_{p, n}$ and  $\gamma_p$ for every $p$ and $n < \omega$


By \cite[1.9]{ef} for every $p$ we have that $\gamma_p= dim_{\mathbb{F}_p}(D(G)[p])$ where $D(G)$ is the divisible part of $G$. Since $\mathbb{Z}(p^{\infty})^{(\lambda)}$ can be purely embedded in $M$, because  $M$ is universal in $\Kabt_\lambda$,  it can be purely embedded in $G$. Hence, $\gamma_p = \lambda$. 

By \cite[1.5]{ef} for every $p$ and $n < \omega$ we have that $\alpha_{p, n} = dim_{\mathbb{F}_p}((p^{n-1}G)[p]/ (p^n G)[p])$. Since $\mathbb{Z}(p^n)^{(\lambda)}$ can be purely embedded in $M$, because $M$ is universal in $\Kabt_\lambda$, it can be purely embedded in $G$. Hence $\alpha_{p, n} = \lambda$. 

Therefore, we can conclude that $M = t(\Pi_p PE(\bigoplus_n \mathbb{Z}(p^n)^{(\lambda)})) \oplus \bigoplus_p \mathbb{Z}(p^\infty)^{(\lambda)}$. 

Observe that $(2)$ follows directly from $(1)$ as $M$ is the direct sum of countably many copies of the $(\lambda, \omega_1)$-limit model by Lemma \ref{ccountablelim}.
\end{proof}

We finish by recording the following results for the class of abelian $p$-groups with pure embeddings. The proofs are similar to those for torsion abelian groups so we omit them. Recall that $G$ is an abelian $p$-group if every element of $G$ has order $p^n$ for some $n \in \mathbb{N}$.

\begin{lemma}
Let $p$ be a fixed prime number and denote by $\Kp$ the class of abelian $p$-groups with pure embeddings.
\begin{enumerate}
\item $\Kp$ is strictly stable.
\item $\Kp$ is $\lambda$-stable if and only if $\lambda^{\aleph_0}=\lambda$.
\item Let $\lambda$ be an infinite cardinal such that $\lambda^{\aleph_0}=\lambda$ and $\alpha < \lambda^+$ be a limit ordinal. If $M$ is a $(\lambda, \alpha)$-limit model in $\Kp$, then:
\begin{itemize}
\item If $\cof(\alpha) > \omega$, then $M \cong t(PE(\bigoplus_n \mathbb{Z}(p^n)^{(\lambda)})) \oplus \mathbb{Z}(p^\infty)^{(\lambda)}$.

\item If $\cof(\alpha)=\omega$, then $M \cong t( PE(\bigoplus_n \mathbb{Z}(p^n)^{(\lambda)}) )^{(\aleph_0)} \oplus \mathbb{Z}(p^\infty)^{(\lambda)}$.
\end{itemize}
\end{enumerate} 
\end{lemma}


\end{document}